\documentclass[11pt]{article}
\usepackage{amsmath, amssymb, amscd, amsthm, amsfonts}
\usepackage{graphicx}
\usepackage{hyperref}
\usepackage{amssymb}
\usepackage{amsmath}
\usepackage{amsthm}
\usepackage{amsfonts}

\usepackage{a4wide}
\usepackage{longtable}
\usepackage{multirow}
\usepackage[latin1]{inputenc}
\usepackage[english]{babel}
\usepackage[T1]{fontenc}
\usepackage{amssymb,amsmath,amsthm,amstext,amsfonts,latexsym}

\usepackage{pgf,tikz}
\usepackage{mathrsfs}
\usetikzlibrary{arrows}

\newtheorem{lemma}{Lemma}[section]
\newtheorem{proposition}{Proposition}[section]
\newtheorem{theorem}{Theorem}[section]

\theoremstyle{definition}
\newtheorem{definition}{Definition}[section]

\theoremstyle{remark}

\usepackage[all]{xy}
\title{On some realizable metabelian $5$-groups}

\date{}

\begin{document}

\maketitle
\begin{center}
{\sc Fouad ELMOUHIB (\textcolor{blue}{the corresponding author}) }\\
{\footnotesize Department of Mathematics and Computer Sciences,\\
Mohammed First University, Oujda, Morocco,\\
\textcolor{blue}{Correspondence: fouad.cd@gmail.com}\\
\vspace{0.7cm}}
{\sc Mohamed TALBI }\\
{\footnotesize Regional Center of Professions of Education and Training,\\
Oujda, Morocco\\
ksirat1971@gmail.com}\\

\vspace{0.7cm}
{\sc Abdelmalek AZIZI }\\
{\footnotesize Department of Mathematics and Computer Sciences,\\
Mohammed First University, Oujda, Morocco,\\
abdelmalekazizi@yahoo.fr}
\end{center}

\begin{abstract}
Let $G$ be a $5$-group of maximal class and $\gamma_2(G) = [G, G]$ its derived group. Assume that the abelianization $G/\gamma_2(G)$ is of type $(5, 5)$ and the transfers $V_{H_1\to \gamma_2(G)}$ and $V_{H_2\to \gamma_2(G)}$ are trivial, where $H_1$ and $H_2$ are two maximal normal subgroups of $G$. Then $G$ is completely
determined with the isomorphism class groups of maximal class. Moreover the group $G$ is realizable with some fields $k$, which is the normal closure of a pure quintic field.
\end{abstract}

\textbf{Key words}: Groups of maximal class, Metabelian $5$-groups, Transfer, $5$-class groups.\\
\textbf{AMS Mathematics Subject Classification}: 11R37, 11R29, 11R20, 20D15.

\section{Introduction}\label{introduction}
The coclass of a $p$-group $G$ of order $p^n$ and nilpotency class $c$ is defined as $cc(G) = n-c$, and a $p$-group $G$ is called of maximal class, if it has $cc(G) = 1$. These groups have been studied by various authors, by determining there classification, the position in coclass graph and the realization of these groups. Blackburn's paper \cite{black}, is considered as reference of the basic materials about these groups of maximal class. Eick and Leendhan-Green in \cite{Erick} gave a classification of $2$-groups. Blackburn's classification in \cite{black}, of the $3$-groups of coclass $1$ implies that these groups exhibit behaviour similar to that proved for $2$-groups. The $5$-groups of maximal class have been investigated in detail in \cite{H.D 1} \cite{H.D 2} \cite{H.D 3} \cite{GRE 2} \cite{Newlam}.\\
With an arbitrary prime $p \geq 2$, let $G$ be a metabelian $p$-group of order $|G| = p^n$ and $cc(G) = 1$, where $n \geq 3$. Then $G$ is of maximal class and the commutator factor group $G/\gamma_2(G)$ of $G$ is of type $(p, p)$ \cite{black}, \cite{Miech}. By $G_a^{(n)}(z,w)$ we denote the representative of an isomorphism class of the metabelian $p$-groups $G$, which satisfies the relations of theorem \ref{theo 2.1}, with a fixed
system of exponents $a, w$ and $z$.\\
In this paper we shall prove that some metabelian $5$-groups are completely determined with the isomorphism class groups of maximal class, furthermore they can be realized.\\ For that we consider $k = \mathbb{Q}(\sqrt[5]{p}, \zeta_5)$, the normal closure of the pure quintic field $\Gamma = \mathbb{Q}(\sqrt[5]{p})$, and also a cyclic Kummer extension of degree $5$ of the $5^{th}$ cyclotomic field $k_0 = \mathbb{Q}(\zeta_5)$, where $p$ is a prime number, such that $p \equiv -1 (\mathrm{mod}\, 25)$. According to \cite{FOU1}, if the $5$-class group of $k$, denoted $C_{k,5}$, is of type $(5,5)$,  we have that the rank of the subgroup of ambiguous ideal classes under the action of $Gal(k/k_0) = \langle \sigma \rangle$, denoted $C_{k,5}^{(\sigma)}$, is rank $C_{k,5}^{(\sigma)} = 1$. Whence by class field theory the relative genus field of the extension $k/k_0$, denoted $k^* = (k/k_0)^*$, is one of the six cyclic quintic extension of $k$.\\ 
By $F_5^{(1)}$ we denote the Hilbert $5$-class field of a number field $F$. Let $G = \mathrm{Gal}\left((k^*)_5^{(1)}/k_0\right)$, we show that $G$ is a metabelian $5$-group of maximal class, and has two maximal normal subgroups $H_1$ and $H_2$, such that the transfers $V_{H_1\to \gamma_2(G)}$ and $V_{H_2\to \gamma_2(G)}$ are trivial. Moreover $G$ is completely
determined with the isomorphism class groups of maximal class. The theoretical results are underpinned by numerical examples obtained with the computational number theory system PARI/GP \cite{PARI}.


\section{PRELIMINARY}
Let $G$ be a metabelian $p$-group of order $p^n$, $n\geq3$, with abelianization $G/\gamma_2(G)$ is of type $(p, p)$, where $\gamma_2(G) = [G,G]$ is the commutator group of $G$. The subgroup $G^p$ of $G$, generated by the $p^{th}$ powers is contained in $\gamma_2(G)$, which therefore coincides with the Frattini subgroups $\phi(G) = G^p\gamma_2(G) = \gamma_2(G)$. According to the basis theorem of Burnside, the group $G$ can thus be generated by two elements $x$ and $y$, $G = <x, y>$. If we declare the lower central series of $G$ recursively by 
\begin{center}
$\begin{cases}
\gamma_1(G) = G\\
\gamma_j(G) = [\gamma_{j-1}(G), G] \,\,\text{for}\,\, j \geq 2,\\
\end{cases}$
\end{center}
Then we have Kaloujnine's commutator relation $[\gamma_j(G), \gamma_l(G)] \subseteq \gamma_{j+l}(G)$, for $j, l \geq 1$, and for an index of nilpotence $c \geq 2$ the series
\begin{center}
$G = \gamma_1(G) \supset \gamma_2(G) \supset.....\supset \gamma_{c-1}(G)\supset \gamma_c(G) = 1$
\end{center}
becomes stationary.\\
The two-step centralizer 
\begin{center}
$\chi_2(G) = \{g \in G \,|\, [g,u] \in \gamma_4(G) \text{for all u} \in \gamma_2(G)\}$
\end{center}
of the two-step factor group $\gamma_2(G)/\gamma_4(G)$, that is the largest subgroup of $G$ such that $[\chi_2(G), \gamma_2(G)]\subset \gamma_4(G)$. It is characteristic, contains the commutator subgroup $\gamma_2(G)$. Moreover $\chi_2(G)$ coincides with $G$ if and only if $n = 3$. For $n \geq 4$, $\chi_2(G)$ is one of the $p+1$ normal subgroups of $G$.\\
Let the isomorphism invariant $k = k(G)$ of $G$, be defined by $[\chi_2(G), \gamma_2(G)]= \gamma_{n-k}(G)$, where $k = 0$ for $n = 3$ and $0 \leq k \leq n-4$ if $n\geq 4$, also for $n \geq p+1$ we have $k = min\{n-4, p-2\}$.\\ 
$k(G)$ provides a measure for the deviation from the maximal degree of commutativity $[\chi_2(G), \gamma_2(G)]= 1$ and is called \emph{defect of commutativity} of $G$.\\
With a further invariant $e$, it will be expressed, which factor 
$\gamma_j(G)/\gamma_{j+1}(G)$ of the lower central series is cyclic for the first time \cite{nebelu}, and we have $e+1 = min\{3 \leq j \leq m \,|\, 1 \leq |\gamma_j(G)/\gamma_{j+1}| \leq p \}$.\\
In this definition of $e$, we exclude the factor $\gamma_2(G)/
\gamma_3(G)$, which is always cyclic. The value $e = 2$ is characteristic for a group $G$ of maximal class.
\subsection{On the $5$-class group of maximal class}
Let $G$ be a metabelian $5$-group of order $5^n,\, n \geq 4$, such that $G/\gamma_2(G)$ is of type $(5,5)$, then $G$ admits six maximal normal subgroups $H_1,...,H_6$, which contain the commutator group $\gamma_2(G)$ as a normal subgroup of index $5$. We have that $\chi_2(G)$ is one of the groups $H_i$. We fix $\chi_2(G) = H_1$. We have the following theorem
\begin{theorem}\label{theo 2.1} Let $G$ be a metabelian $5$-group of order $5^n,\, n\geq 4$, with the abelianization $G/
\gamma_2(G)$ is of type $(5,5)$ and $k = k(G)$ its invariant defined before. Assume that $G$ is of maximal class, then $G$ can be generated by two elements, $G = < x, y >$, be selected such that $x \in G\setminus \chi_2(G)$ and $y \in \chi_2(G)\setminus \gamma_2(G)$. Let $s_2 = [y,x] \in \gamma_2(G)$ and $s_j = [s_{j-1}, x] \in \gamma_j(G)$ for $j \geq 3$. Then we have:
\item[$(1)$] $s_j^{5}s_{j+1}^{10}s_{j+2}^{10}s_{j+3}^{5}s_{j+4} = 1$ for $j \geq 2$.
\item[$(2)$] $x^5 = s_{n-1}^{w}$ with $w \in \{0, 1, 2, 3, 4\}$.
\item[$(3)$] $y^{5}s_{2}^{10}s_{3}^{10}s_{4}^{5}s_{5} = s_{n-1}^z$ with $z \in \{0, 1, 2, 3, 4\}$.
\item[$(4)$] $[y, s_2] = \prod\limits_{i=1}^k s_{n-i}^{a_{n-i}}$ with $a = (a_{n-1},...a_{n-k})$ exponents such that $0 \leq a_{n-i} \leq 4$.
\end{theorem}

\begin{proof}
See [\cite{Miech}, Theorem 2] for $p = 5$.
\end{proof}
The six maximal normal subgroups $H_1....H_6$ are arranged as follow:\\
$H_1 = \langle y, \gamma_2(G)\rangle = \chi_2(G)$, $H_i = \langle xy^{i-2}, \gamma_2(G)\rangle$ for $2 \leq i \leq 6$.
The order of the abelianization of each $H_i$, for $1 \leq i \leq 6$, is given by the following theorem.
\begin{theorem}\label{theo 2.2}
Let $G$, $H_i$ and the invariant $k$ as before. Then for $1 \leq i \leq 6$, the order of the commutator factor groups of $H_i$ is given by:
\item[$(1)$] If $n = 2$ we have : $|H_i/\gamma_2(H_i)| = 5$ for $1 \leq i \leq 6$.
\item[$(2)$] If $n \geq 3$ we have : $|H_i/\gamma_2(H_i)| = 5^2$ for $2 \leq i \leq 6$, and $|H_1/\gamma_2(H_1)| = 5^{n-k-1}$
\end{theorem}

\begin{proof}
See [\cite{Mayer secon}, Theorem 3.1] for $p = 5$.
\end{proof}

\begin{lemma}\label{lemma 2.1}
Let $G$ be a $5$-group of order  $|G| = 5^n,\, n\geq 4$. Assume that the commutator group $G/\gamma_2(G)$ is of type $(5, 5)$. Then $G$ is of maximal class if and only if $G$ admits a maximal normal subgroup with factor commutator of order $5^2$. Furthermore $G$ admits at least five maximal normal subgroups with factor commutator of order $5^2$.
\end{lemma}

\begin{proof} 
Assume that $G$ is of maximal class, then by theorem \ref{theo 2.2}, we conclude that $G$ has five maximal normal subgroups with the order of commutator factor is $5^2$ if $n \geq 4$, and has six when $n = 3$. Conversely, Assume that $cc(G) \geq 2$, the invariant $e$ defined before is greater than 3, and since each maximal normal subgroup $H$ of $G$ verify $|H/\gamma_2(H)| \geq 5^{e}$ we get that $|H/\gamma_2(H)| > 5^{2}$
\end{proof}

\subsection{On the transfer concept}
Let $G$ be a group and let $H$ be a subgroup of $G$. The transfer from $G$ to $H$ can be decomposed as follows: Also we note $\bar{V}$ instead of $V_{G\to H}$.

\begin{center}
\begin{tikzpicture}
\begin{scope}[xscale=2,yscale=2]
  
  \node (A) at (0,1) {$G$};
  \node (B) at (2,1){$H/\gamma_2(H)$ };
  \node (H)   at (0,0) {$G/\gamma_2(G)$ };
  \node (I)   at (1,-0.5) {\underline{Figure 1}: Transfer diagram};
  \draw [<-,>=latex] (B) -- (A)   node[midway,above,rotate=-40] {};
  
  \draw [->,>=latex] (H) -- (B)   node[midway,below right] {$\bar{V}$};
    \draw [->,>=latex] (A) -- (H)   node[midway,below right] {};

\end{scope}

\end{tikzpicture}

\end{center}

\begin{definition}
Let $G$ be a group, $H$ be a normal subgroup of $G$, and let $g \in G$ such that, $f$ is the order of $gH$ in $G/H$, $r = \frac{[G:H]}{f}$
 and $g_1,....g_r$ be a representative system of $G/H$, then the transfer from $G$ to $H$, noted $V_{G\to H}$, is defined by:\\
 \begin{center}
 $V_{G\to H}$ :  $G/\gamma_2(G) \longrightarrow H/\gamma_2(H)$\\
 \hspace{3.5cm}$g\gamma_2(G) \longrightarrow \prod_{i=1}^r g_i^{-1}g^fg_i\gamma_2(H)$ 
\end{center}
\end{definition}
In the special case that $G/H$ is cyclic group of order $5$ and $G = \langle h, H \rangle$, then the transfer $V_{G\to H}$ is given as:

 $(1)$ If $g \in H$; then $V_{G\to H}(g
\gamma_2(G)) = g^{1+h+h^2+h^3+h^4}
\gamma_2(H)$\\

 $(2)$  $V_{G\to H}(h\gamma_2(G)) = h^5\gamma_2(H)$

\section{MAIN RESULTS}
In this section we investigate the purely group theoretic results to determine the invariants of metabelian $5$-group of maximal class developed in theorem \ref{theo 2.1}. Furthermore we show that a such metabelian $5$-group is realized by the Galois group of some fields tower.

\subsection{Invariants of metabelian $5$-group of maximal class}
In this paragraph, we keep the same hypothesis on the group $G$ and the generators $G = \langle x, y \rangle$, such that $x \in G \setminus \chi_2(G)$ and $y \in \chi_2(G)\setminus \gamma_2(G)$. The six maximal normal subgroups of $G$ are as follows: $H_1 = \chi_2(G) = \langle y, \gamma_2(G) \rangle$ and $H_i = \langle xy^{i-2}, \gamma_2(G) \rangle$ for $2 \leq i \leq 6$.\\
In the case that the transfers from two subgroups $H_i$ and $H_j$ to $\gamma_2(G)$ are trivial, we can determine completely the $5$-group $G$.
\begin{proposition}\label{prop 3.1}
Let $G$ be a metabelian $5$-group of maximal class of order $5^n,\, n \geq 4$. If the transfers $V_{\chi_2(G)\to \gamma_2(G)}$ and $V_{H_2\to \gamma_2(G)}$  are trivial, then $n \leq 6$ and $\gamma_2(G)$ is of exponent $5$. Furthermore:\\
- If $n = 6$ then $G \sim G_a^{(6)}(1, 0)$ where $a = 0$ or $1$.\\
- If $n = 5$ then $G \sim G_a^{(5)}(0, 0)$ where $a = 0$ or $1$.\\
- If $n = 4$ then $G \sim G_0^{(4)}(0, 0)$.
\end{proposition}
\begin{proof}
Assume that $n\geq 7$, then $\gamma_5(G) = \langle s_5, \gamma_6(G)\rangle$, because $G$ is of maximal class and $|\gamma_5(G)/\gamma_6(G)| = 5$. By [\cite{black}, lemma 3.3] we have $y^5s_5 \in \gamma_6(G)$, thus $\gamma_5(G) = \langle s_5^4, \gamma_6(G)\rangle  = \langle y^5s_5s_5^4, \gamma_6(G)\rangle  = \langle y^5, \gamma_6(G)\rangle$, and since $V_{\chi_2(G)\to\gamma_2(G)}(y) = y^5 = 1$, because the transfers are trivial by hypothesis, we get that\\ 
$\gamma_5(G) = \gamma_6(G)$, which is impossible, whence $n \leq 6$ and According to [\cite{black}, lemma 3.2], $\gamma_2(G)$ is of exponent $5$.\\
If $n = 6$, we have $V_{\chi_2(G)\to\gamma_2(G)}$ and $V_{H_2\to\gamma_2(G)}$ are trivial, so by theorem \ref{theo 2.1} we obtain $x^5 = s_{5}^w = 1$ which imply $w = 0$, because $0 \leq w \leq 4$. Since $\gamma_2(G)$ is of exponent $5$, we have $s_2^5 = 1$ and by theorem \ref{theo 2.1} the relation $s_4^{5}s_5^{10}s_6^{10}s_7^5s_8 = 1$ gives $s_4^5 = 1$, also $s_3^{5}s_4^{10}s_5^{10}s_6^5s_7 = 1$ gives $s_3^5 = 1$. We replace in $y^5s_2^{10}s_3^{10}s_4^{5}s_5 = s_5^z$ and we get $s_5 = s_5^z$, whence $z = 1$. We have $[\chi_2(G), \gamma_2(G)] \subset \gamma_{6-k}(G) \subset \gamma_4(G)$ then $6-k \geq 4$, and $0 \leq k \leq 2$, thus $[y, s_2] = s_4^{\alpha\beta}$, $a = (\alpha, \beta)$. If $k = 0$, then $a = 0$ and $G \sim G_0^{(6)}(1, 0)$, if $k = 1$ then $a = 1$ and $G \sim G_1^{(6)}(1, 0)$ and if $k = 2$ then $G \sim G_a^{(6)}(1, 0)$.\\
If $n = 5$, we have $[\chi_2(G), \gamma_2(G)] \subset \gamma_{5-k}(G) \subset \gamma_4(G)$ then $5-k \geq 4$, and $0 \leq k \leq 1$. We have $s_4^5 = 1$, $s_2^5 = s_3^5 = 1$ and $[y, s_2] = s_4^a$. the relation $y^5s_2^{10}s_3^{10}s_4^{5}s_5 = s_4^z$ imply $s_4^z = 1$ so $z = 0$. As $n =6$ we obtain $w = 0$. If $k = 0$ then $G \sim G_0^{(5)}(0, 0)$ and if $k = 1$  $G \sim G_a^{(5)}(0, 0)$.\\
If $n = 4$, Since $[\chi_2(G), \gamma_2(G)] \subset \gamma_{5-k}(G) \subset \gamma_4(G)$ we have $4-k \geq 4$, and $k = 0$, thus $[y, s_2] = 1$, i.e $a = 0$. By the same way in this case we have $w = z = 0$, therefor $G \sim G_0^{(4)}(0, 0)$.
\end{proof}

\begin{proposition}\label{prop 3.2} Let $G$ be a metabelian $5$-group of maximal class of order $5^n$. If the transfers $V_{H_2\to \gamma_2(G)}$ and $V_{H_i\to \gamma_2(G)}$, $3 \leq i \leq 6$, are trivial, then we have:\\
- If $n = 5$ or $6$ then $G \sim G_a^{(n)}(0, 0)$.\\
- If $n \geq 7$ then $G \sim G_0^{(n)}(0, 0)$ .
\end{proposition}
\begin{proof}
If $n = 5$ or $6$, by [\cite{black}, theorem 1.6] we have $[\chi_2(G), \gamma_2(G)] = 1$ and $[\chi_2(G), \gamma_2(G)] \subset \gamma_4(G)$ elementary, and $(\gamma_2(\chi_2(G)))^5 = 1$ and $\prod_{i=2}^3[\gamma_i(G), \gamma_4(G)] = 1$, we conclude that $(xy)^5 = x^5y^5s_2^{10}s_3^{10}s_4^{5}s_5$ and we have $y^5s_2^{10}s_3^{10}s_4^{5}s_5 = s_{n-1}^z$ then $(xy)^5 = x^5s_{n-1}^z$ and since 
$V_{H_2\to\gamma_2(G)}$ and $V_{H_3\to\gamma_2(G)}$ are trivial then $(xy)^5 = x^5 = s_{n-1}^z = s_{n-1}^w = 1$, thus $z = w = 0$. Since $[\chi_2(G), \gamma_2(G)] = \gamma_{n-k} \subset \gamma_4(G)$ we have $n-k \geq 4$, whence $0 \leq k \leq 2$ because $n = 5$ or $6$ then $G \sim G_a^{(n)}(0, 0)$.\\
If $n \geq 7$, according to corollary page 69 of \cite{black} we have, 
$(\gamma_j(\chi_2(G)))^5 = \gamma_{j+4}(G)$ for $j \geq 2$, and since $y^5s_2^{10}s_3^{10}s_4^{5}s_5 = s_{n-1}^z$ we obtain:\\ 
\centerline{$y^5 = s_{n-1}^zs_5^{-1}s_4^{-1}s_3^{-10}s_2^{-10} \equiv s_{n-1}^zs_5^{-1} \,\mathrm{mod \gamma_6(G)}$}\\ because $s_2^5 \in \gamma_6(G)$, $s_3^5 \in \gamma_6(G)$ and $s_4^5 \in \gamma_6(G)$, and since $n\geq 7$ we have $s_{n-1} \in \gamma_6(G)$, therefor $V = V_{H_3\to\gamma_2(G)}(y) \equiv s_5^{-1}\, \mathrm{mod \gamma_6(G)}$. Thus Im$(V) \subset \gamma_5(G)$, In fact Im$(V) = \gamma_5(G)$, and also we have $y \notin \mathrm{ker}(V)$ and $\forall f \geq 2$  $y^ks_f^l \notin \mathrm{ker}(V)$. The kernel of $V$ is formed by elements of $\gamma_2(G)$ of exponent $5$, its exactly $\gamma_{n-4}(G)$, and since $G$ is of maximal class then the rank of $\gamma_2(G)$ is $2$ and $\gamma_2(G)$ admits exactly $25$ elements of exponent $5$, these elements form $\gamma_{n-4}(G)$. We conclude that $|\chi_2(G)/\gamma_2(\chi_2(G))| = |\gamma_{n-4}(G)|\times |\gamma_5(G)| = 5^4\times 5^{n-5} = 5^{n-1} = |\chi_2(G)|$, whence $\chi_2(G)$ is abelian because $\gamma_2(\chi_2(G)) = 1$, consequently $[y, s_2] = 1$, thus $a = 0$. As the cases $n = 5$ or $6$ we obtain $(xy)^5  = x^5s_{n-1}^z$, therefor $z = w = 0$, hence $G \sim G_0^{(n)}(0, 0)$.\\
In the case when $V_{H_2\to\gamma_2(G)}$ and $V_{H_i\to\gamma_2(G)}$, $4 \leq i \leq 6$ are trivial, according to [\cite{black}, theorem 1.6] we have $(xy^\mu)^5 = x^5(y^5s_2^{10}s_3^{10}s_4^{5}s_5)^\mu = s_{n-1}^ws_{n-1}^{\mu z}$ with $\mu = 2, 3, 4$, then we can admit the same reasoning to prove the result.
\end{proof}

\begin{proposition}\label{prop 3.3} Let $G$ be a metabelian $5$-group of maximal class of order $5^n$. If the transfers $V_{H_i\to\gamma_2(G)}$ and $V_{H_j\to\gamma_2(G)}$, where $i, j \in \{3, 4, 5, 6\}$ and $ i \neq j$, are trivial, then we have: $G \sim G_0^{(n)}(0, 0)$.
\end{proposition}
\begin{proof}
Assume that $H_i = \langle xy^{\mu_1}, \gamma_2(G)\rangle$ and $H_j = \langle xy^{\mu_2}, \gamma_2(G)\rangle$ where $\mu_1, \mu_2 \in \{1, 2, 3, 4\}$ and $\mu_1 \neq \mu_2$. According to [\cite{black}, theorem 1.6] we have already prove that $(xy^{\mu_1})^5 = s_{n-1}^{w+\mu_1 z}$ and $(xy^{\mu_2})^5 = s_{n-1}^{w+\mu_2 z}$. Since $V_{H_i\to\gamma_2(G)}$ and $V_{H_j\to\gamma_2(G)}$ are trivial, we obtain $s_{n-1}^{w+\mu_1 z} = s_{n-1}^{w+\mu_2 z} = 1$ then $w+\mu_1 z \equiv w+\mu_2 z \equiv 0\, (\mathrm{mod}\, 5)$ and since $5$ does not divide $\mu_1-\mu_2$ we get $z = 0$ and at the same time $w = 0$. To prove $a = 0$ we admit the same reasoning as proposition \ref{prop 3.2}.
\end{proof}

\subsection{APPLICATION}
Through this section we denote by:
\begin{itemize}
\item[-] $p$ a prime number such that $p \equiv -1 (\mathrm{mod}\, 25)$.
\item[-] $k_0 = \mathbb{Q}(\zeta_5)$ the $5^{th}$ cyclotomic field, $(\zeta_5 = e^{\frac{2\pi i}{5}}).$
\item[-] $k = k_0(\sqrt[5]{p})$ a cyclic Kummer extension of $k_0$ of degree $5$.
\item[-] $C_{k,5}$ the $5$-ideal class group of $k$.
\item[-] $k^* = (k/k_0)^*$ the relative genus field of $k/k_0$.
\item[-] $F_5^{(1)}$ the absolute Hilbert $5$-class field of a number field $F$.
\item[-] $G =\mathrm{Gal}\left((k^*)_5^{(1)}/k_0\right)$.
\end{itemize}
We begin by the following theorem.
\begin{theorem}\label{theo 3.1}
Let $k = \mathbb{Q}(\sqrt[5]{p}, \zeta_5)$ be the normal closure of a pure quintic field $\mathbb{Q}(\sqrt[5]{p})$, where $p$ a prime congruent to $-1$ modulo 25. Let $k_0$ be the the $5^{th}$ cyclotomic field. Assume that the $5$-class group $C_{k,5}$ of $k$, is of type $(5, 5)$, then $Gal(k^*/k_0)$ is of type $(5, 5)$, and two sub-extensions of $k^*/k_0$ admit a trivial $5$-class number.
\end{theorem}

\begin{proof}
By $C_{k,5}^{(\sigma)}$ we denote the subgroup of ambiguous ideal classes under the action of $Gal(k/k_0) = \langle \sigma \rangle$. According to [\cite{FOU1}, theorem 1.1], in this case of the prime $p$ we have rank $C_{k,5}^{(\sigma)} = 1$, and by class field theory, since $[k^* : k] = |C_{k,5}^{(\sigma)}|$, we have that $k^*/k$ is a cyclic quintic extension, whence $Gal(k^*/k_0)$ is of type $(5,5)$.\\
Since $p \equiv -1 (\mathrm{mod}\, 25)$, then $p$ splits in $k_0$ as $p = \pi_1\pi_2$, where $\pi_1,\, \pi_2$ are primes of $k_0$. By [\cite{Mani}, theorem 5.15] we have explicitly the relative genus field $k^*$ as $k^* = k(\sqrt[5]{\pi_1^{a_1}\pi_2^{a_2}}) = k_0(\sqrt[5]{\pi_1\pi_2}, \sqrt[5]{\pi_1^{a_1}\pi_2^{a_2}})$ with $a_1, a_2 \in \{1, 2, 3, 4\}$ such that $a_1 \neq a_2$. Its clear that the extension $k^*/k_0$ admits six sub-extensions, where $k$ is one of them, and the others are $k_0(\sqrt[5]{\pi_1^{a_1}\pi_2^{a_2}})$, $k_0(\sqrt[5]{\pi_1^{a_1+1}\pi_2^{a_2+1}})$, $k_0(\sqrt[5]{\pi_1^{a_1+2}\pi_2^{a_2+2}})$, $k_0(\sqrt[5]{\pi_1^{a_1+3}\pi_2^{a_2+3}})$ and $k_0(\sqrt[5]{\pi_1^{a_1+4}\pi_2^{a_2+4}})$. Since $a_1, a_2 \in \{1, 2, 3, 4\}$, we can see that the extensions $L_1 = k_0(\sqrt[5]{\pi_1})$ and $L_2 = k_0(\sqrt[5]{\pi_2})$ are sub-extensions of $k^*/k_0$.\\
In [\cite{Mani}, section 5.1], we have an investigation of the rank of ambiguous classes of $k_0(\sqrt[5]{x})/k_0$, denoted $t$. We have $t = d+q^*-3$, where $d$ is the number of prime divisors of $x$ in $k_0$, and $q^*$ an index defined as [\cite{Mani}, section 5.1]. For the extensions $L_i/k_0$, $(i = 1, 2)$, we have $d = 1$ and by [\cite{Mani}, theorem 5.15] we have $q^* = 2$, hence $t = 0$.\\
By $h_5(L_i)$, $(i = 1, 2)$, we denote the class number of $L_i$, then we have $h_5(L_1) = h_5(L_2) = 1$. Otherwise $h_5(L_i) \neq 1 $, then there exists an unramified cyclic extension of $L_i$, denoted $F$. This extension is abelian over $k_0$, because $[F : k_0] = 5^2$, then $F$ is contained in $(L_i/k_0)^*$ the relative genus field of $L_i/k_0$. Since $[(L_i/k_0)^* : L_i] = 5^t = 1$, we get that $(L_i/k_0)^* = L_i$, which contradicts the existence of $F$. Hence the $5$-class number of $L_i$, $(i = 1, 2)$, is trivial.
\end{proof}
In what follows, we denote by $L_1$ and $L_2$ the two sub-extensions of $k^*/k_0$, which verify theorem \ref{theo 3.1}, and by $\tilde{L}$ the three remaining sub-extensions different to $k$. Let $G = Gal((k^*)_5^{(1)}/k_0)$, we have $\gamma_2(G) = Gal((k^*)_5^{(1)}/k^*)$, then $G/\gamma_2(G) = Gal(k^*/k_0)$ is of type $(5, 5)$, therefore $G$ is metabelian $5$-group with factor commutator of type $(5, 5)$, thus $G$ admits exactly six maximal normal subgroups as follows:
\centerline{$H = Gal((k^*)_5^{(1)}/k)$, $H_{L_i} = Gal((k^*)_5^{(1)}/L_i)$, $(i = 1, 2)$, $\tilde{H} = Gal((k^*)_5^{(1)}/\tilde{L})$}.
With $\chi_2(G)$ is one of them.\\
Now we can state our principal result.

\begin{theorem}\label{theo 3.2}
Let $G = Gal((k^*)_5^{(1)}/k_0)$ be a $5$-group of order $5^n$, $n\geq 4$, then $G$ is a metabelian of maximal class. Furthermore we have:\\
- If $\chi_2(G) = H_{L_i} (i = 1, 2)$ then: $G \sim G_a^{(n)}(z,0)$ with $n \in \{4, 5, 6\}$ and $a, z\in \{0, 1\}$.\\
- If $\chi_2(G) = \tilde{H}$ then : $G \sim G_1^{(n)}(0,0)$ with $n = 5$ or $6$.\\
\centerline{\textcolor{white}{............} $G \sim G_0^{(n)}(0,0)$ with $n \geq 7$ such that $n = s+1$ where $h_5(\tilde{L}) = 5^s$.}
\end{theorem}
\begin{proof}
Let $G = Gal((k^*)_5^{(1)}/k_0)$ and $H = Gal((k^*)_5^{(1)}/k)$ its maximal normal subgroup, then $\gamma_2(H) = Gal((k^*)_5^{(1)}/k_5^{(1)})$, therefor $H/\gamma_2(H) = Gal(k_5^{(1)}/k) \simeq C_{k,5}$, and as $C_{k,5}$ is of type $(5, 5)$ by hypothesis we get that $|H/\gamma_2(H)| = 5^2$. Lemma \ref{lemma 2.1} imply  that $G$ is a metabelian $5$-group of maximal class, generated by two elements $G = \langle x, y \rangle$, such that, $x \in G \setminus \chi_2(G)$ and $y \in \chi_2(G) \setminus \gamma_2(G)$. Since $\chi_2(G) = \langle y , \gamma_2(G) \rangle$, we have $\chi_2(G) \neq H$. Otherwise we get that $|H/\gamma_2(H)| = 5^2$ which contradict theorem  \ref{theo 2.1}.\\
According to theorem \ref{theo 3.1}, we have $h_5(L_1) = h_5(L_2) = 1$ then the transfers $V_{H_{L_i}\to \gamma_2(G)}$ are trivial.\\
If $\chi_2(G) = H_{L_i}$ the results are nothing else than  proposition \ref{prop 3.1}.\\
If $\chi_2(G) = \tilde{H}$ and $n = 4$ then $\gamma_4(G) = 1$ and $[\chi_2(G), \gamma_2(G)] = \gamma_2(\tilde{H})$, also $[\chi_2(G), \gamma_2(G)] = \gamma_4(G) = 1$ then $\chi_2(\tilde{H}) = 1$, whence $\tilde{H}$ is abelian. Consequently $\tilde{H}/\gamma_2(\tilde{H}) = C_{\tilde{L}, 5}$, so $h_5(\tilde{L}) = |\tilde{H}| = 5^3$ because its a maximal subgroup of $G$. Since $\tilde{L}$ and $k$ have always the same conductor, we deduce that $h_5(k)$ and $h_5(\tilde{L})$ verify the relations 
$5^5h_{\tilde{L}} = uh_\Gamma^4$ and $5^5h_{k} = uh_\Gamma^4$, given by C. Parry in \cite{Parry}, where $u$ is a unit index and a divisor of $5^6$. Using the $5$-valuation on these relations we get that $h_5(\tilde{L}) = 5^s$ where $s$ is even,  which contradict the fact that $h_5(\tilde{L}) = 5^3$, hence $n \geq 5$.\\
The results of the theorem are exactly application of propositions \ref{prop 3.2}, \ref{prop 3.3}. According to proposition \ref{prop 3.2}, if $n \geq 7$ we have $|\chi_2(G)| = 5^{n-1}$ and since $h_5(\tilde{L}) = |\tilde{H}/\gamma_2(\tilde{H})| = |\tilde{H}| = 5^{n-1} = 5^s$ we deduce that $n = s+1$.
\end{proof}

\section{Numerical examples}
For these numerical examples of the prime $p$, we have that $C_{k,5}$ is of type $(5, 5)$ and rank $C_{k,5}^{(\sigma)}\,=\,1$, which mean that $k^*$ is cyclic quintic extension of $k$, then by theorem \ref{theo 3.2} we have a completely determination of $G$. We note that the absolute degree of $(k^*)_5^{(1)}$ surpass $100$, then the task to determine the order of $G$ is definitely far beyond the reach of computational algebra systems like MAGMA and PARI/GP.
\begin{center}
Table 1: $k\,=\,\mathbb{Q}(\sqrt[5]{p}, \zeta_5)$ with $C_{k,5}$ is of type $(5,5)$ and rank $C_{k,5}^{(\sigma)}\,=\,1$.\\

\begin{tabular}{|c|c|c|c|c|}
\hline 
$p$ &  $p\,(\mathrm{mod}\,25)$ & $h_{k,5}$ & $C_{k,5}$ & rank $(C_{k,5}^{(\sigma)})$ \\ 
\hline 
149 &  -1 & 25 & $(5,5)$ & 1 \\ 
199 &  -1 & 25 & $(5,5)$ & 1 \\ 
349 &  -1 & 25 & $(5,5)$ & 1 \\ 
449 &  -1 & 25 & $(5,5)$ & 1 \\ 
559 &   -1 & 25 & $(5,5)$ & 1 \\ 
1249 &   -1 & 25 & $(5,5)$ & 1 \\ 
1499 &   -1 & 25 & $(5,5)$ & 1 \\ 
1949 & -1 & 25 & $(5,5)$ & 1 \\ 
1999 &   -1 & 25 & $(5,5)$ & 1 \\ 
2099 &   -1 & 25 & $(5,5)$ & 1 \\ 
\hline 
\end{tabular} 
\end{center}

\newpage


\end{document}